\documentclass[12pt]{amsart}
\usepackage[margin=1.2in]{geometry}
\usepackage{amssymb,euscript,amsmath, mathrsfs,amscd}
\usepackage{tikz}
\usepackage{stmaryrd}
\usepackage{graphics}

\newtheorem{theorem}{Theorem}[section]

\newtheorem{proposition}[theorem]{Proposition}
\newtheorem{lemma}[theorem]{Lemma}
\newtheorem{corollary}[theorem]{Corollary}
\theoremstyle{definition}
\newtheorem{definition}[theorem]{Definition}
\newtheorem{example}[theorem]{Example}
\newtheorem{notation}[theorem]{Notation}
\newtheorem{remark}[theorem]{Remark}
\numberwithin{equation}{section}
\author{Federico Castillo}

\address{Pontificia Universidad Cat\'olica de Chile.
E-mail: federico.castillo@mat.uc.cl}
\begin{document}
\title{A pithy look at the Polytope Algebra}
\maketitle

\begin{abstract}
This is a hands on introduction to McMullen's Polytope Algebra. More than interesting on its own, this algebra was McMullen's tool to give a combinatorial proof of the g-theorem.
\end{abstract}

\section{Introduction}

One of the biggest achievements in polytope theory was the complete characterization of $f$-vectors of simplicial polytopes, known as the $g$-theorem (for a precise statement see \cite[Chapter 3]{green}). In the early 70's Peter McMullen conjectured a set of necessary and sufficient conditions and less than a decade later Lou Billera and Carl Lee proved sufficiency, whereas Richard Stanley proved necessity. For an interesting survey about the history of this result we refer the reader to \cite{gteo}.

Stanley's proof in \cite{necessity} imports a result from complex geometry, namely the Hard-Lefschetz theorem, so for a while people attempted to find a more combinatorial proof, one that used more elementary arguments. McMullen succeeded initially in \cite{simple} and then corrected and simplified parts of it in \cite{weights}. The heart of the $g$-theorem lies in associating, to each simplicial polytope $P$, a finite dimensional algebra with particular properties. By now there are many different ways of doing this. For instance, in \cite{billera} Billera uses piecewise polynomial functions on a fan (an approach also explored by Brion \cite{brion}), and in \cite{toric} Fulton-Sturmfels use the Chow ring of the associated toric variety which they show is isomorphic to the Minkowski weights, which is the same ring that McMullen employs in \cite{weights} although he constructs it from his earlier work in \cite{main}.

Here we give an overview of the Polytope Algebra as developed in \cite{main}. Of all the approaches mentioned above, this one is of the easiest to define and in a way the most ``polytopal". Apart from the intrinsic interest, we also wanted to introduce, in a natural way, the notion of Minkowski weights which continues to play an important role in research nowadays, for example in tropical geometry  \cite{tropical}.

\section*{Acknowledgments}
The author is deeply grateful to Raman Sanyal from whom he learned most of the material. These notes combines the original material straight from \cite{main} together with Sanyal lectures in the MSRI Summer School \emph{Positivity Questions in Geometric Combinatorics} in summer 2017.
The author also thanks Takayuki Hibi and Akiyoshi Tsuchiya for the kind hospitality and the organization of the \emph{Summer Workshop on Lattice Polytopes 2018}.

\section{Polytopes and their faces}

\begin{definition}
A \emph{polytope} $P\subset \mathbb{R}^d$ is the convex hull of finitely many points $\{\textbf{v}_1.\cdots,\textbf{v}_m\}$. More precisely,
\[
\displaystyle P=\textrm{conv}\{\textbf{v}_1,\cdots,\textbf{v}_m\}=\left\{\sum_{i=1}^m\lambda_i\textbf{v}_i: \lambda_1+\cdots+\lambda_m=1, \lambda_1,\cdots,\lambda_m\geq0\right\}.
\]
\end{definition}

An inclusion minimal set of points whose convex hull is $P$ is called the vertex set of $P$, $\textrm{vert}(P)$. Such a minimal set exists and is unique. 

\begin{theorem}[Minkowski-Weyl]
Every polytope $P$ is the bounded intersection of finitely many halfspaces. More precisely, there exist $\textbf{a}_1,\cdots,\textbf{a}_n\in\mathbb{R}^d\backslash\{0\}$ and  $b_1,\cdots,b_n\in\mathbb{R}$ such that
\[
\displaystyle P = \left\{ \textbf{x}\in\mathbb{R}^d: \textbf{a}_1^t\textbf{x}\leq b_1, \cdots, \textbf{a}_n^t\textbf{x}\leq b_n\right\}.
\]
\end{theorem}
As before, there may be redundant inequalities. An inclusion minimal set of inequalities that define $P$ is called the set of facet inequalities. Such a minimal set exists and is unique.

\begin{definition}
Let $P\subset \mathbb{R}^d$ be a polytope. We define the following notions.
\begin{enumerate}
\item The \emph{affine hull} of $P$: $\textrm{aff}(P)=\bigcap \left\{L\subset \mathbb{R}^d\textrm{ affine space such that }P\subset L \right\}$.
\item The \emph{linear span} of $P$: $\textrm{lin}(P)=$ the linear subspace that is parallel to $\textrm{aff}(P)$.
\item The \emph{relative interior} of $P$: $\textrm{relint}(P)=$ the interior of $P$ relative to $\textrm{aff}(P)$.
\item The \emph{dimension} of $P$: $\dim(P)=\dim \textrm{lin}(P)$.
\end{enumerate}
\end{definition}

\textbf{Basic information:} From now on we are going to repeatedly use the following notation. $P$ is a $d$-dimensional polytope, or $d$-polytope, with $n$ facet inequalities (or just facets for short) and $m$ vertices. The set of all polytopes in $\mathbb{R}^d$ (not necessarily full dimensional) is denoted ${\mathcal{P}}_d$ and $\hat{{\mathcal{P}}_d}:={\mathcal{P}}_d\backslash \{\emptyset\}$.

\subsection{Faces of polytopes.}

For any $\textbf{c}\in\mathbb{R}^d$ (which we will think of as a linear functional) define
\begin{align*}
P^\textbf{c} :&= \{\textbf{x}\in P: \textbf{c}^t\textbf{x}\geq \textbf{c}^t\textbf{y}, \forall \textbf{y}\in P\},\\
&=\{\textbf{x}\in P: \textbf{c}^t\textbf{x}=\delta\},\quad \delta = \underset{\textbf{x}\in P}{\max}\hspace{5pt} \textbf{c}^t\textbf{x}.
\end{align*}
Subsets of $P$ of the above form are called \emph{faces} of $P$. If $P=\textrm{conv}\{\textbf{v}_1,\cdots,\textbf{v}_m\}$ is the vertex description of $P$, then $P^\textbf{c} = \textrm{conv}\{\textbf{v}_i: \textbf{c}^tv_i = \delta\}$, so faces are polytopes themselves. Moreover, there can be only finitely many faces. By convention, $\emptyset$ and $P$ itself are faces.
\begin{remark}
In what follows we will only consider face directions $\textbf{c}$ such that $|\textbf{c}|=1$.
\end{remark}

\begin{definition}
Let $P$ be a polytope and let $\mathcal{F}_k(P)$ be the set of $k$-dimensional faces. Without a subscript $\mathcal{F}(P)$ is the set of all faces of $P$. Equipped with the partial order given by containment, $\mathcal{F}(P)$ is called the \emph{face lattice} of $P$. This captures the combinatorial information of $P$. Two polytopes are said to be \emph{combinatorially equivalent} if they have isomorphic face lattices.
\end{definition}

\begin{definition}
A weaker invariant is the \emph{f-vector}: $f(P)=(f_0,f_1,\cdots,f_{d-1},f_d)$. Where $f_i$ is the number of $i$-dimensional faces of $P$ (recall that faces are themselves polytopes, hence they have dimension). Often people add $f_{-1}=1$ for the empty set. Also note that $f_d=1$ always. The 0-faces are the vertices, 1-faces the edges, and $(d-1)$-faces are the facets.
\end{definition}

We say a $d$-polytope is \emph{simple} if each vertex is contained in exactly $d$ facets. 

\section{Valuations.}

For any set $S\subset \mathbb{R}^d$ denote by $[S]$ the \emph{indicator function} from $\mathbb{R}^d$ to $\{0,1\}$ defined as
\[
[S](p)=\begin{cases} 1, \quad p\in S,\\ 0,\quad p\notin S.\end{cases}
\]

\begin{definition}
A function $f:\mathbb{R}^d\longrightarrow \mathbb{Z}$ is a \emph{polytopal simple function} if it can be written as
\[
f = \alpha_1[Q_1]+\cdots+\alpha_k[Q_k], \quad\alpha_i\in \mathbb{Z},
\]
with $Q_1,\cdots,Q_k$ polytopes.

\end{definition}

\begin{notation}
Recall that ${\mathcal{P}}_d = \{\textrm{ polytopes }\subset \mathbb{R}^d\}$ and now we define $\mathcal{SP}_d := \{\textrm{ polytopal simple functions }\}$.
Often we won't distinguish between a polytope and its indicator function.
\end{notation}

By definition, indicator functions of ${\mathcal{P}}_d$ span $\mathcal{SP}_d$ but they are far from a basis, they satisfy relations like $[P\cup Q]=[P]+[Q]-[P\cap Q]$.

\begin{definition}
A \emph{valuation} is any group homomorphism $\phi: \mathcal{SP}_d\longrightarrow G$ with $G$ an abelian group.
\end{definition}

The first and most fundamental valuation is the \emph{Euler characteristic} (See \cite[Theorem 2.4]{barvi}).

\begin{theorem}
There exists a unique valuation $\chi:{\mathcal{P}}_d\longrightarrow \mathbb{Z}$, called the \emph{Euler characteristic} with
\begin{itemize}
\item $\chi(\emptyset)=0$.
\item $\chi(P)=\begin{cases} 1,\quad P\in\hat{{\mathcal{P}}_d}\\ 0,\quad \textrm{else.}\end{cases}$
\end{itemize}
\end{theorem}

Even though $\mathcal{SP}_d$ is spanned by indicator functions of polytopes, it also contains indicator functions of other non polytopal sets. For instance, the Euler formula is equivalent to the relation

\begin{lemma}
Let $P$ be a $d$-polytope then
\begin{equation*}\label{eq:relintp}
\displaystyle [P^\circ] = \sum_{F\subseteq P} (-1)^{\textrm{dim}(P)-\textrm{dim}(F)}[F].
\end{equation*}
Furthermore $\chi\left([P^\circ]\right)=(-1)^{\textrm{dim}(P)}$.
\end{lemma}

Indicators functions come with a natural product, namely pointwise multiplication, but we consider a different product that gives the following ring structure.
\begin{proposition}
The map $\ast:\mathcal{SP}_d\times \mathcal{SP}_d\to \mathcal{SP}_d$ defined by $[P]\ast[Q]:=[P+Q]$ gives $\mathcal{SP}_d$ the structure of a commutative ring. The multiplicative unit is $[\textbf{0}]$, the indicator function at the origin.
\end{proposition}
With this operation we gain some interesting inverses.
\begin{proposition}
Let $P\in\hat{{\mathcal{P}}_d}$, then $[P]$ is invertible:
\begin{equation}
[P]^{-1} = (-1)^{\dim P}[-\textrm{relint}(P)].
\end{equation}
\end{proposition}
The next proposition is fundamental for what follows. 
\begin{proposition}\label{thm:importantid}
If $P=\textrm{conv}\{v_1,\cdots,v_m\}$ is a polytope, then 
\begin{equation}\label{eq:importantid}
\left([P]-[v_1]\right)*\left([P]-[v_2]\right)*\cdots *\left([P]-[v_m]\right)=0.
\end{equation}
\end{proposition}
Let's first illustrate an example.
\begin{example}\label{ex:1dimho}
Let $P$ be a one dimensional polytope $[a,b]\subset \mathbb{R}$. Then $([P]-[a])*([P]-[b])=[2P]-[P+a]-[P+b]+[a+b]=0$, as can be seen by inspection in Figure \ref{fig:onedim}. 
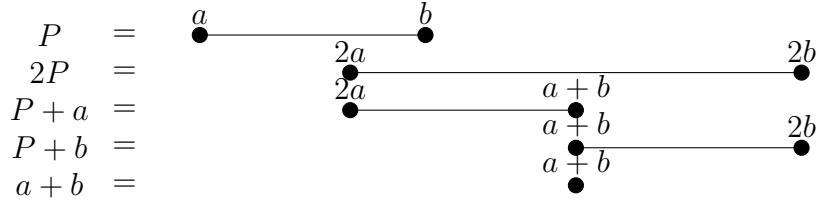
\begin{figure}[h]
\begin{tikzpicture}
\node at (0,1) {$P$};
\node at (0,0.5) {$2P$};
\node at (0,0) {$P+a$};
\node at (0,-0.5) {$P+b$};
\node at (0,-1) {$a+b$};

\node at (1,1) {$=$};
\node at (1,0.5) {$=$};
\node at (1,0) {$=$};
\node at (1,-0.5) {$=$};
\node at (1,-1) {$=$};

\draw (2,1)--(5,1);
\node[above] at (2,1) {$a$};
\node[above] at (5,1) {$b$};
\draw[fill] (2,1) circle [radius=0.1];
\draw[fill] (5,1) circle [radius=0.1];

\draw (2*2,0.5)--(2*5,0.5);
\node[above] at (2*2,1/2) {$2a$};
\node[above] at (5*2,1/2) {$2b$};
\draw[fill] (2*2,1/2) circle [radius=0.1];
\draw[fill] (5*2,1/2) circle [radius=0.1];

\draw (2+2,0)--(5+2,0);
\node[above] at (2+2,0) {$2a$};
\node[above] at (5+2,0) {$a+b$};
\draw[fill] (2+2,0) circle [radius=0.1];
\draw[fill] (5+2,0) circle [radius=0.1];

\draw (2+5,-0.5)--(5+5,-0.5);
\node[above] at (2+5,-0.5) {$a+b$};
\node[above] at (5+5,-0.5) {$2b$};
\draw[fill] (2+5,-0.5) circle [radius=0.1];
\draw[fill] (5+5,-0.5) circle [radius=0.1];

\node[above] at (2+5,-1) {$a+b$};
\draw[fill] (2+5,-1) circle [radius=0.1];
\end{tikzpicture}

\caption{An illustration of Proposition \ref{thm:importantid}}\label{fig:onedim}
\end{figure}

\end{example}

\begin{proof}[Sketch of proof]
We first prove it for $P=\Delta_{m-1}=\textrm{conv}\{\mathbf{e}_1,\cdots,\mathbf{e}_m\}=\{\textbf{x}\in\mathbb{R}^m: x_i\geq 0, \quad \sum x_i = 1\}.$ We define $\mathbf{e}_I := \sum_{i\in I}\mathbf{e}_i$ for $I\subset [m]$. We must show that
\begin{equation}\label{eq:importantidproof}
\displaystyle \sum_{I\subset[m]}(-1)^{|I|}\left[(m-|I|)\Delta_{m-1}+\mathbf{e}_I\right] = 0.
\end{equation}

For $q\in \mathbb{R}^m$ how do we decide if $q\in \left[(m-|I|)\Delta_{m-1}+\mathbf{e}_I\right]$? This is equivalent to $q-\mathbf{e}_I\in (m-|I|)\Delta_{m-1}$. This happens if and only if
\[
q-\mathbf{e}_I \geq 0\qquad \sum q_i = m,
\]
where the first inequality is coordinate wise. Now define $I_0:=\{i: q_i\geq 1\}$. For a fixed $q$ with $\sum q_i = m$ the value of the function in the left hand side of Equation \eqref{eq:importantidproof} is
\[
\displaystyle \sum_{I\subset I_0}(-1)^{|I|} = \sum_{j=0}^{|I_0|} \binom{I_0}{j}(-1)^j = (1-1)^{|I_0|} = \begin{cases}0,\quad I_0\neq \emptyset,\\ 1,\quad I_0=\emptyset, \end{cases}
\]
but $I_0$ will never be empty since we are asumming $\sum q_i = m$. This finishes the proof for the standard simplex.

\par For the general case we can write $P=\pi(\Delta_{m-1})$, where $\pi:\mathbb{R}^m\longrightarrow \mathbb{R}^d$, defined by mapping $\mathbf{e}_i\longrightarrow v_i$. Then one can argue that $\pi$ induces a map $\pi_*:\mathcal{SP}_m\to\mathcal{SP}_d$ sending $\pi_*[P]=[\pi(P)]$ for $P\in {\mathcal{P}}_m$ and such that the linear relations are preserved.
\end{proof}

\section{The Polytope Algebra}

We define $\mathcal{T}:= \mathbb{Z}\{[P+t]-[P]\quad \forall P\in {\mathcal{P}}_d, t\in\mathbb{R}^d\}.$ This is actually an \emph{ideal} of $\mathcal{SP}_d$, called the translation ideal. 
\begin{definition}
The polytope algebra is defined as
\[
\Pi^d := \mathcal{SP}_d/ \mathcal{T}.
\]
We denote by $\llbracket P\rrbracket$ the class of $[P]$ in $\Pi^d$.
\end{definition}
The multiplicative identity is $1:=\llbracket \{\textbf{0}\}\rrbracket$, the class of the origin as a zero dimensional polytope.
\begin{definition}
We define also \emph{dilation} maps which are in fact endormorphism. Let $D_\lambda:\Pi^d\to\Pi^d$ be defined in the generators as $D_\lambda\llbracket P\rrbracket:=\llbracket \lambda P\rrbracket$ for $\lambda\in\mathbb{R}_{\geq0}$. Notice that $D_\lambda$ and $D_{\lambda^{-1}}$ are inverses to each other for $\lambda>0$.
\end{definition}
\begin{example}
Let's begin by understanding $\Pi^1$. This is generated by \emph{integer} combinations of segments. For each line segment $[a,b]$ we have $\llbracket [a,b]\rrbracket=\llbracket [a,b)\rrbracket+\llbracket \{b\}\rrbracket=\llbracket [0,b-a)\rrbracket+\llbracket \{0\}\rrbracket$, since we can decompose and translate the pieces. Notice that the sum of two classes of half open segments can be represented again by a half open segment; $\llbracket [0,r)\rrbracket+\llbracket [0,s)\rrbracket=\llbracket [0,r)\rrbracket+\llbracket [r,s+r)\rrbracket=\llbracket [0,s+r)\rrbracket$.\\

It is not true that $\llbracket [0,s)\rrbracket\ast\llbracket [0,r)\rrbracket=\llbracket [0,s+r)\rrbracket$, since the definition of the product with Minkowski sums only is intended for the generators, the indicator functions of polytopes. Indeed we have
\[
\llbracket [0,s)\rrbracket\ast\llbracket [0,r)\rrbracket=\left(\llbracket [0,s]\rrbracket-\llbracket \{0\}\rrbracket\right)\ast\left(\llbracket [0,r]\rrbracket-\llbracket \{0\}\rrbracket\right),
\]
which we already saw in Example \ref{ex:1dimho} to be zero.

With this is mind, $\Pi^1\cong \mathbb{Z}\oplus\mathbb{R}$ with multiplication defined by $(a,b)\cdot(a',b')=(aa',ab'+a'b)$.
\end{example}

Now that we understand $\Pi^1$, we move on to the general case. The following is the main structural result.

\begin{theorem}\label{thm:structure}
The polytope algebra is a graded ring, generated in degree 1.
$\Pi^d=\Pi_0\oplus \Pi_1\oplus\cdots \oplus \Pi_d$. Furthermore
\begin{itemize}
\item[(i)] $\Pi_0\cong \mathbb{Z}$.
\item[(ii)] $\Pi_i$ is an $\mathbb{R}$-vector space for $i>0$.
\item[(iii)] $\Pi_d\cong \mathbb{R}$.
\end{itemize}
\end{theorem}
The theorem means that $\Pi^d$ is almost an $\mathbb{R}$-algebra, except for the fact that $\Pi_0\cong \mathbb{Z}$. Notice that $\Pi^d$ comes from $\mathcal{SP}_d$ which is made of \emph{integer} combinations of indicator functions, so the fact that we end up with a $\mathbb{R}$ action should be surprising, in fact we will see below that scaling can be complicated. As a first step in proving Theorem \ref{thm:structure} we have the following proposition.

\begin{proposition}
The map $\Pi^d\longrightarrow \mathbb{Z}$ induced by $\chi$ allows us to decompose $\Pi^d = \mathbb{Z}\oplus \Pi_+$ as a direct sum, where $\Pi_+:= \textrm{ker } \chi$.
\end{proposition}

\begin{proof}
Decompose $x=\sum \alpha_i\llbracket P_i\rrbracket\in \Pi^d$ as 
\begin{equation}\label{eq:chi}
x=\sum \alpha_i\chi(P_i)\cdot 1+\sum \alpha_i\left(\llbracket P_i\rrbracket-1\right)=\chi(x)\cdot1+\left(x-\chi(x)\cdot1\right).
\end{equation}
\end{proof}

From Equation \ref{eq:chi} we have that $\Pi_+ = \mathbb{Z}\{\llbracket P\rrbracket - 1:\quad P\in\hat{{\mathcal{P}}}_d \}$. Proposition \ref{thm:importantid} in the polytope algebra means that all those generators are nilpotent. More precisely:

\begin{corollary}\label{cor:exp}
For $P\in \hat{{\mathcal{P}}}_d$ we have $\Big(\llbracket P\rrbracket -1\Big)^{r}=0$ in $\Pi^d$ for $r>d$. 
\end{corollary}
\begin{proof}
From Theorem \ref{thm:importantid} we know that $\Big(\llbracket P\rrbracket -1\Big)^{f_0(P)}=0$. Now we argue that we can lower the exponent. Notice that $\llbracket nP\rrbracket=\llbracket P+P\cdots+P\rrbracket=\llbracket P\rrbracket^n$ so we can write
\begin{equation}\label{eq:weakexp}
\displaystyle \llbracket nP\rrbracket=\left(1+\left(\llbracket P\rrbracket-1\right)\right)^n=\sum_{k=0}^n\binom{n}{k}\left(\llbracket P\rrbracket-1\right)^k.
\end{equation}
Every polytope can be triangulated which means we can always write, through inclusion-exclusion, $\llbracket P\rrbracket=\sum \alpha_i\llbracket T_i\rrbracket$, where each $T_i$ is a simplex of dimension $\leq d$ (and hence with at most $d+1$ vertices) and $\alpha_i\in\{-1,+1\}$. When dilating $P$ we can dilate each piece so that 
\begin{equation}\label{eq:poly}
\llbracket nP\rrbracket=\sum \alpha_i\llbracket nT_i\rrbracket.
\end{equation}
Expanding the right hand side with an analogous relation to Equation \ref{eq:weakexp} for each $T_i$, we get a polynomial in $n$ with coefficients in $\Pi^d$ of degree at most $d$, since all terms $\left(\llbracket T_i\rrbracket-1\right)^r$ vanish for $r>d$. This means that the right hand side of Equation \ref{eq:weakexp} is also a polynomial in $n$ of degree $\leq d$ and hence $\left(\llbracket P\rrbracket-1\right)^r$ vanish for $r>d$.
\end{proof}
\begin{theorem}\label{thm:vectorspace}
The abelian group $\Pi_+$ is a $\mathbb{Q}$-vector space.
\end{theorem}
This is a nontrivial statement. We need to make sense of $\frac{1}{m}x$ for $x\in\Pi_+$, more precisely, we need that for every $m\in\mathbb{Z}_{>0}$ there exists a unique $h\in\Pi_+$ with $x=m\cdot h$. We prove existence and uniqueness in two separate lemmas. Together they prove Theorem \ref{thm:vectorspace}.
\begin{lemma}\label{lem:divisible}
The abelian group $\Pi_+$ is divisible. For every $x\in \Pi_+$ and $m\in \mathbb{Z}_{>0}$, there exist $h\in \Pi_+$ such that $m\cdot h = f$. 
\end{lemma}

\begin{remark}
It is important to keep in mind that if $P$ is a polytope, then $2\llbracket P\rrbracket$ is \textbf{not} equal to $\llbracket 2P \rrbracket$, the indicator of the second dilation of $P$. One quick way to remember this is to apply Euler characteristic: $\chi\left(2\llbracket P\rrbracket\right)=2$ whereas $\chi\left(\llbracket 2P\rrbracket\right)=1$.
\end{remark}

\begin{proof}
The following proof is indirect and dry. To see how to actually divide see Example \ref{ex:divide}.
It is enough to show the result for $m>1$ prime. Consider $N=m^e>d+1$ a large power of $m$. Then we have
\[
\llbracket P\rrbracket - 1 = \llbracket \frac{1}{N}P\rrbracket^N - 1 = \sum_{i=1}^d\binom{N}{i}\left(\llbracket \frac{1}{N}P\rrbracket-1\right)^i.
\]
Now since $m$ is prime and $N$ is a power of $m$, we have that $m$ divides all the binomial coefficients.
\end{proof}

\begin{example}
Notice that in the real line it is straighforward to divide half open segments, since $m\cdot [0,r/m) = [0,r)$ up to translation.
\end{example}

\begin{example}\label{ex:divide}
Now sketch an example in two dimensions. Again we will not divide a whole polytope, but a polytope minus a point. Since all points are equivalent under translation, we choose to remove a vertex.
\begin{figure}[h]
\centering
\includegraphics[scale=0.8]{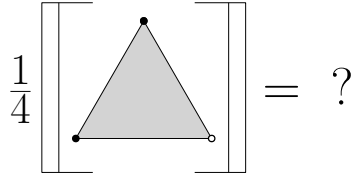}
\caption{We need to find an integer combination of polytopes $h$ such that $\llbracket P \rrbracket=4h$}
\end{figure}\\
The idea now is that we are going to decompose that simplex in a convenient way. Along the way we will get 4 copies of its $\frac{1}{4}$-dilation, as may be expected, but we also get a number of products of simplices of strictly smaller dimensions. Then by induction in dimension we can divide each of them. See Figure \ref{fig:exdiv}.
\begin{figure}[h]
\centering
\includegraphics[scale=0.8]{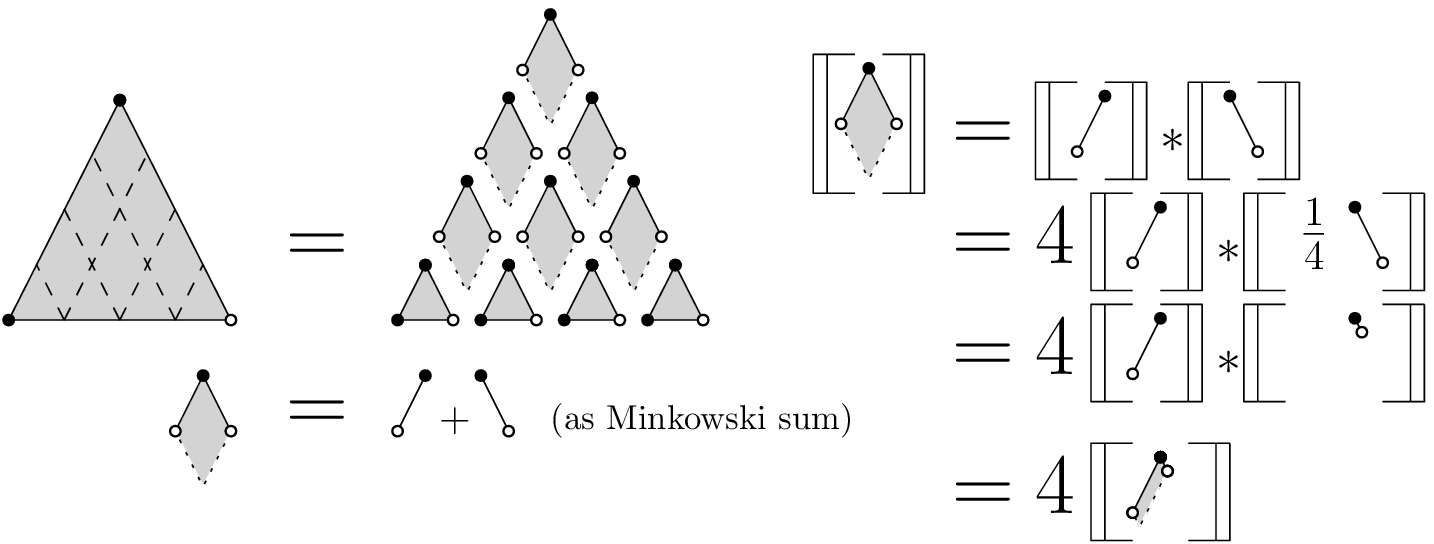}
\includegraphics[scale=0.8]{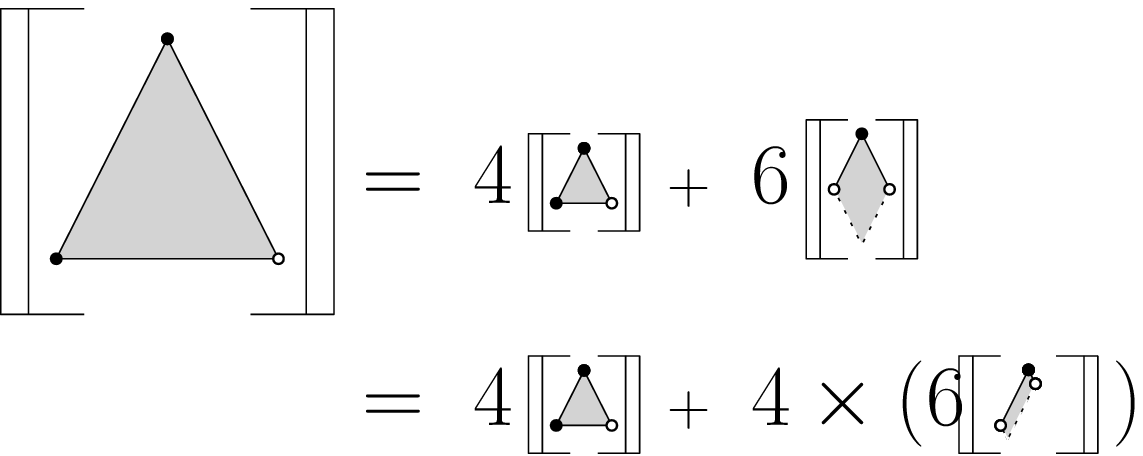}
\caption{An example of division by 4.}\label{fig:exdiv}
\end{figure}
\end{example}

\begin{lemma}\label{lem:torsion}
The abelian group $\Pi_+$ has no torsion elements.
\end{lemma}
\begin{proof}
Consider the following filtration $\Pi_+=Z_1\supset Z_2\supset\cdots\supset Z_d\supset Z_{d+1}$ where $Z_r$ is generated by elements of the form $\left(\llbracket P\rrbracket-1\right)^j$ for $j\geq r$. The proof of the lemma follows from two observations.

The first one is that $D_\lambda\left(\llbracket P\rrbracket-1\right)^r=\left(\llbracket \lambda P\rrbracket-1\right)^r$, since $D_\lambda$ is a ring endomorphism, so it commutes with taking powers. This implies that $D_\lambda Z_r\subset Z_r$ and $Z_r\subset D_{\lambda^{-1}}Z_r$.

The second observation is that if $x\in Z_r$, then 
\begin{equation}\label{eq:lem}
D_nx-n^rx\in Z_{r+1},
\end{equation}
 for $n$ a natural number. It is enough to check it on the generators. We apply Equation \ref{eq:weakexp} 
\begin{equation}\label{eq:strongexp}
D_n\left(\llbracket P\rrbracket-1\right)=\left(\llbracket nP\rrbracket-1\right)= \binom{n}{1}\left(\llbracket P\rrbracket-1\right)+\binom{n}{2}\left(\llbracket P\rrbracket-1\right)^2+\cdots.
\end{equation}
Raising the above expression to the $r$ power we get that $D_n\left(\llbracket P\rrbracket-1\right)^r-n^r\left(\llbracket P\rrbracket-1\right)^r\in Z_{r+1}$. Since it is true for the generators then Equation \ref{eq:lem} holds for all $Z_r$.

Now we finish the proof of the lemma. Let $x\in Z_r$ for some $r\geq1$, and $nx=0$ for a natural $n$. Then $D_nx=D_nx-n^rx$ since $nx=0$ and hence $D_nx\in Z_{r+1}$ by \ref{eq:lem} and also $x\in D_{n^{-1}}Z_{r+1}\subset Z_{r+1}$. This implies that $x\in Z_j$ for $j>>0$, but since they are eventually zero, $x$ must be zero.

\end{proof}

Combining Lemmas \ref{lem:torsion} and \ref{lem:divisible}, we get Theorem \ref{thm:vectorspace}.
Now that we can make sense of \emph{rational multiples} of element in $z \in\Pi_+$, we can define the exponential and the logarithm as formal power series with coefficients in $\mathbb{Q}$.
\[
\displaystyle \exp(z):= \sum_{k\geq 0} \frac{1}{k!} z^k,\qquad \log(1+z) = \sum_{k\geq 1} \frac{(-1)^{k-1}}{k} z^k.
\]
The usual identities formally apply
\begin{align*}
\log \exp (z) = \exp \log (z) = z,\\
\exp(a+b) = \exp(a)\cdot \exp(b),\\
\log(a\cdot b)= \log(a) + \log(b).\\
\end{align*}

With this in mind, we can define
\[
\displaystyle \log(\llbracket P\rrbracket) = \log(1+\left(\llbracket P\rrbracket-1\right)) = \sum_{k\geq 1}^d \frac{(-1)^{k-1}}{k}\left(\llbracket P\rrbracket-1\right)^k,
\]
which is a finite sum since $\left(\llbracket P\rrbracket-1\right)^i=0, i>d$. 

\begin{theorem}\label{thm:polynomiality}
For $P\in{\mathcal{P}}_d$, define $p:=\log\left(\llbracket P\rrbracket\right)$, then 
\begin{equation}\label{eq:polyid}
\llbracket P\rrbracket^n=\llbracket nP\rrbracket = 1+pn+\frac{1}{2}p^2n^2+\cdots+\frac{1}{d!}p^dn^d.
\end{equation}

\end{theorem}
\begin{proof}
We simply manipulate our expressions:
\begin{equation}\label{eq:preehr}
\llbracket nP\rrbracket = \llbracket P\rrbracket^n = \exp\left(\log \llbracket P\rrbracket^n\right)= \exp\left(np\right)= \sum_{i=0}^d \frac{1}{i!}p^in^i
\end{equation}
\end{proof}
A consequence of Theorem \ref{thm:polynomiality} is that for $n=1$:
\begin{equation}\label{eq:graded}
\llbracket P\rrbracket = 1+p+\frac{1}{2}p^2+\cdots+\frac{1}{d!}p^d.
\end{equation}
which is going to be our graded decomposition.

\begin{definition}
Define $\Pi_k:= \mathbb{Z}\{\frac{1}{k!}\log\left(\llbracket P\rrbracket\right)^k\quad P\in{\mathcal{P}}_d\}$.
\end{definition}
With this we now have $\Pi = \Pi_0+\Pi_1+\cdots +\Pi_d$. However we want this sum to be direct. For this we shall use our dilation maps. We have:
\[
D_r(\log\llbracket P\rrbracket) = \log D_r(\llbracket P\rrbracket) = \log (\llbracket rP\rrbracket)=\log (\llbracket P\rrbracket^r) = r\log(\llbracket P\rrbracket).
\]
Moreover we have
\[
x\in\Pi_k \Longleftrightarrow D_rx=r^kx\quad r\in\mathbb{Q}_{>0}.
\]

\begin{proposition}
We have $\Pi = \Pi_0\oplus\Pi_1\oplus\cdots \oplus\Pi_d$ as a direct sum.
\end{proposition}
\begin{proof}
Suppose there exist $x_i\in\Pi_i$ with $x_0+x_1+\cdots+x_d=0$. From this, we conclude $D_N(x_0+x_1+\cdots+x_d)=0$ which is the same as $x_0+x_1N^1+x_2N^2+\cdots x_dN^d=0$. This last equality is true for all $N>0$ hence we can conclude that $x_i=0.$
\end{proof}

\begin{proposition}
The decomposition $\Pi = \Pi_0\oplus\Pi_1\oplus\cdots \oplus\Pi_d$ gives a standard grading.
\end{proposition}
\begin{proof}
What we need to prove is that $x\in \Pi_i, y\in\Pi_j$ imply $x\ast y\in\Pi_{i+j}$. This follows from the fact that $D_r$ is a ring map,
\[
D_r(x\ast y)=D_r(x)\ast D_r(y) = r^ix\ast r^jy = r^{i+j}(x\ast y).
\]
\end{proof}

\begin{corollary}\label{cor:homog}
For any $\phi:{\mathcal{P}}_d\longrightarrow G$ translation invariant valuation, $\phi$ is homogeneous of degree $k$ (i.e., $\phi(nP)=n^k$ for $P\in{\mathcal{P}}_d$) if and only if $\phi(\Pi_j)=0$ for $j\neq k$. Also, for any translation invariant valuation $\phi$, homogeneous or not, we can uniquely decompose it $\phi=\phi_0+\phi_1+\cdots+\phi_d$ in homogeneous parts.
\end{corollary}

The volume is the \emph{unique} (up to a multiple) translation invariant valuation of degree $d$ on ${\mathcal{P}}_d$ (See \cite[Section 7]{main}).

\begin{corollary}\label{cor:volume}
The volume valuation induces an isomorphism $\textrm{Vol}_d: \Pi_d\to \mathbb{R}$.
\end{corollary}

We can convince ourselves that $\Pi_d$ is not trivial. In fact, the class of each half open segment is in $\Pi_1$.  And hence $\llbracket [\textbf{0},\textbf{e}_1)\rrbracket\ast\cdots \ast\llbracket[\textbf{0},\textbf{e}_d)\rrbracket\in \Pi_d$. Such a class can be represented by the half open cube $\{\textbf{x}\in\mathbb{R}^d: 0\leq x_i<1\textrm{ for all }i\in[d]\}$ which has volume one, so it can be taken a the generator of $\Pi_d$.

\begin{example}\label{ex:vol}
Corollary \ref{cor:volume} implies that any two elements in $\Pi_d$ with the same volume are equivalent in $\Pi$. In Figure \ref{fig:vol} we illustrate one example.
\begin{figure}[h]
\centering
\includegraphics[scale=0.75]{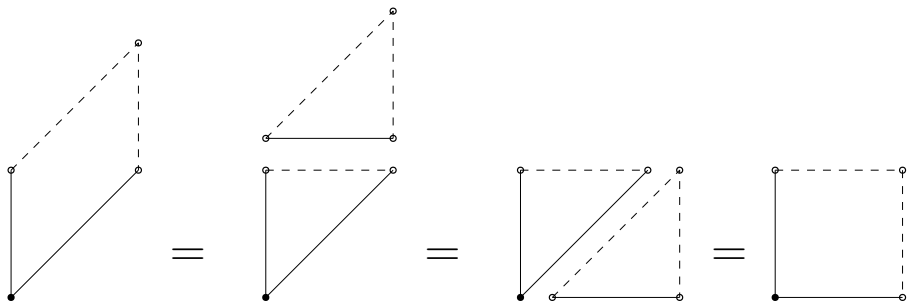}
\caption{An instace of two half open parallelograms with the same area.}\label{fig:vol}
\end{figure}
\end{example}

\subsection{Two applications}

\subsubsection{Mixed Volumes}

For any polytope $P\in {\mathcal{P}}_d$ we have that the volume is an homogeneous valuation of degree $d$, i.e., $\textrm{Vol}(tP)=t^d\textrm{Vol}(P)$. Here is a more general version.

\begin{theorem}\label{thm:mixvol}
For polytopes $P_1,\cdots, P_m\in {\mathcal{P}}_d$, we have
\[
\displaystyle \textrm{Vol}(\lambda_1P_1+\lambda_2P_2+\cdots+\lambda_mP_m) = \sum_{i_1,\cdots,i_d=1}^d V(P_{i_1},\ldots,P_{i_d}) \lambda_{i_1}\cdots\lambda_{i_d},
\]
where each symmetric coefficient $V(P_{i_1},\ldots,P_{i_d})$ depends only on the bodies
$P_{i_1},\ldots,P_{i_d}$.
\end{theorem}

\begin{proof}

In the polytope algebra, consider the element $\llbracket P_1\rrbracket^{\lambda_1}\ast\llbracket P_2\rrbracket^{\lambda_2}\ast\cdots\ast\llbracket P_m\rrbracket^{\lambda_m}$, where, for now, the $\lambda$'s are integers. 
Using Equation \eqref{eq:polyid} we get
\begin{equation}\label{eq:mixpoly}
\displaystyle \llbracket P_1\rrbracket^{\lambda_1}\ast\llbracket P_2\rrbracket^{\lambda_2}\ast\cdots\ast\llbracket P_m\rrbracket^{\lambda_m} = \prod_{i=1}^m \sum_{j=1}^d \frac{1}{d!}p_i^j\lambda_i^j,
\end{equation}
where as usual $p_i=\log\llbracket P_i\rrbracket$. Taking $\textrm{Vol}_d$ at both sides we get precisely
\[
\displaystyle \textrm{Vol}(\lambda_1P_1+\lambda_2P_2+\cdots+\lambda_mP_m) = \sum_{i_1,\cdots,i_d=1}^d \textrm{Vol}_d(p_{i_1}\ast\cdots\ast p_{i_d}) \lambda_{i_1}\cdots\lambda_{i_d}.
\]
Notice that $\textrm{Vol}_d$ is homogeneous of degree $d$ so we only need to keep track of the degree $d$ part of the right hand side. Hence we can define $V(P_{i_1},\ldots,P_{i_d})=\textrm{Vol}_d(p_{i_1}\ast\cdots\ast p_{i_d})$ to finish the proof.
\end{proof}

\begin{definition}
The function $V(P_{i_1},\ldots,P_{i_d})$ is the \emph{mixed volume} of the tuple of polytopes $P_{i_1},\ldots,P_{i_d}\in{\mathcal{P}}_d$. 
\end{definition}

\subsubsection{Ehrhart Polynomial}
Let's focus briefly on \emph{lattice} polytopes and \emph{lattice} invariant valuations for some lattice $\Lambda\subset \mathbb{R}$. The main example of a valuation invariant under lattice translation (but not under all translations) is the counting valuation, $E(P):=|P\cap\Lambda|$. This case is substantially different and we cannot directly apply our results. However we can prove the following classic theorem. 
\begin{theorem}[Ehrhart]
The function $E_P(n):=E(nP)$ agrees with a polynomial in $n$ of degree $d$ whenever $n\in\mathbb{Z}_{\geq 0}$.
\end{theorem}
\begin{proof}
Notice that at least Equation \eqref{eq:weakexp} doesn't involve any scaling (since the binomial coefficients are integers), it is invariant under lattice translation, and all polytopes appearing are lattice polytopes as long as $P$ is. Then one can apply the counting valuation $E(P):=|P\cap\Lambda|$ on both sides to obtain the following (we only sum up to $d$ by Corollary \ref{cor:exp}):
\begin{equation}\label{eq:ehrhart}
E_P(n):=E(nP)=\sum_{i=0}^d\binom{n}{i}\tilde{f^*_i}(P),
\end{equation}
where $\tilde{f^*_i}(P)$ are some integers depending on $P$. Notice that for $n=0$ we indeed get $E_P(0)=1$.
\end{proof}

\subsection{Minkowski Weights}

Corollary \ref{cor:volume} does not say that $\textrm{Vol}_d(P)=\textrm{Vol}_d(Q)$ for polytopes $P$ and $Q$ implies $\llbracket P\rrbracket=\llbracket Q\rrbracket$, since such elements do not belong to $\Pi_d$ (see Figure \ref{fig:vol} for an example of what it does say).
Nevertheless, we have the following theorem. For an modern elementary proof see \cite{klain}.

\begin{theorem}[Minkowski]\label{thm:mink}
Let $P,Q\in {\mathcal{P}}_d$ be two \emph{full} dimensional polytopes. Then $
\textrm{Vol}_{d-1}(P^\textbf{c}) = \textrm{Vol}_{d-1}(Q^\textbf{c})$ for all $c$ implies that $P$ and $Q$ are equal up to translation.
\end{theorem}

A priori we need to check infinitely many directions, but it could be finite if we know where to look (there are finite $\textbf{c}$ that give facets, so most of the time both quantities are zero). We need to generalize a bit previous theorem so that we have a criterion for lower dimensional polytopes.

\begin{definition}
A $(d-k)$ \emph{frame} is a $(d-k)$ tuple of vectors $U=(\mathbf{u}_1,\mathbf{u}_2,\cdots,\mathbf{u}_{d-k})$ in $\mathbb{R}^d$ such that $\mathbf{u}_i^t\mathbf{u}_j=\delta_{ij}$. Given a $(d-k)$ frame $u$ we define the face $P^U~:=\left(\cdots\left(\left(P^{\mathbf{u}_1}\right)^{\mathbf{u}_2}\right)\cdots\right)^{\mathbf{u}_{d-k}}$. Because of orthogonality, the dimension is reduced by at least 1 on each step, so $\dim(P^u)\leq k$

\end{definition}

We also define the \emph{frame functionals} to be $V_U(P):=\textrm{Vol}_k(P^U)$. These are homogeneous valuations of degree $k$.

\begin{theorem}[Generalized Minkowski]\label{thm:genmink}
Let $P,Q\in {\mathcal{P}}_d$ be two $k$-polytopes. Then $
V_U(P) = V_U(Q)$ for all $(d-k)$-frame functionals $U$ implies that $P$ and $Q$ are equal up to translation.
\end{theorem}

\section{A finitely generated subalgebra.}

So far $\Pi^d$ is not finitely generated as an algebra. We will restrict to a subalgebra.

\begin{definition}
For $P,Q\in{\mathcal{P}}_d$ we say $Q$ is a \emph{Minkowski summand} of $P$, and we write $P\leq Q$ if there exists $R$ such that $P=Q+R$. We say $Q$ is a \emph{weak Minkowski summand} of $P$, and we write $P\preceq Q$ if there exists $\lambda\in\mathbb{R}_{>0}$ such that $Q\leq \lambda P$.
\end{definition}

\begin{definition}
Fixing $P\in{\mathcal{P}}_d$, we define $\Pi(P)=\mathbb{Z}\{\llbracket Q\rrbracket: Q\preceq P\}$.
\end{definition}

\begin{proposition} The following statements hold.
\begin{itemize}
\item[(i)] $\Pi(P)$ is a \emph{finitely generated} graded subalgebra of $\Pi$.
\item[(ii)] $Q\preceq P\Longrightarrow \Pi(Q)\subset \Pi(P)$.
\item[(iii)] $\Pi(P)$ is generated by Minkowski summands of $P$.
\item[(iv)] $\Pi(P)+\Pi(Q)\subset \Pi(P+Q)$.
\end{itemize}
\end{proposition}
\begin{proof}
It is a subalgebra because the Minkowski sum of two weak summands is still a weak summand. For finite generation see Remark $\ref{rem:fin}$ below. It is graded since $Q\in\Pi(P)$ implies $\log(Q)\in \Pi(P)$. The other conditions are not hard to check.
\end{proof}
There is a criterion for determining when a polytope is a summand of another (See \cite[Chapter 15]{grunbi}).
\begin{theorem}[Shephard]\label{thm:shepard}
The polytope $Q$ is a Minkowski summand of $P$ if and only if the following conditions are satisfied:
\begin{enumerate}
\item[(i)] $\dim P^c \geq \dim Q^c$ for any $c\in\mathbb{R}^d$.
\item[(ii)] If for some $c\in\mathbb{R}^d$ we have $\dim P^c = 1$, then $vol_1(P^c)\geq \textrm{vol}_1(Q^c)$.
\end{enumerate}
\end{theorem}

\begin{corollary}\label{cor:wms}
The polytope $Q$ is a weak Minkowski summand of $P$ if and only if $\dim P^c \geq \dim Q^c$ for any $c\in\mathbb{R}^d$.
\end{corollary}

\begin{remark}
The condition in Corollary \ref{cor:wms} is equivalent to saying that $Q\preceq P$ if and only if the normal fan of $P$ refines the normal fan of $Q$. From that, it is not hard to show that for any polytope $P$ one can find a \emph{simple} polytope $P'$ with $P\preceq P'$, and since $\Pi(P)\subset \Pi(P')$ we can always assume that $P$ is a simple polytope.
\end{remark}
Corollary \ref{cor:wms} is crucial in turning the infinite conditions of Theorem \ref{thm:genmink} into a finite set of conditions. 

\begin{definition}
Fix a simple polytope $P$ together with frame functionals $U(F)$ for $F\in\mathcal{F}(P)$ such that $P^U=F$.
Define the \emph{Minkowski} map $\phi$ as:
\begin{align}
\phi: \Pi(P) &\longrightarrow \bigoplus_{i=0}^d \mathbb{R}^{f_i(P)},\\
\llbracket Q \rrbracket&\longrightarrow \left(V_{U(F)}(Q)\right)_{F\in\mathcal{F}(P)},
\end{align}
\end{definition}

Theorem \ref{thm:genmink} guarantees that this map is an injection. However it is not surjective. The frame functionals satisfy linear relations.

\begin{theorem}[Minkowski Relations]\label{thm:minkrel}
Given a polytope $P$ with unit facet normals $\mathbf{u}_1,\cdots,\mathbf{u}_n$, the following linear equation holds:
\begin{equation}\label{eq:minkrel}
\displaystyle \sum_{i=1}^n \mathbf{u}_i\textrm{Vol}_{d-1}(P^{\mathbf{u}_i}) = 0.
\end{equation}
\end{theorem}
\begin{proof} We choose a generic direction $\mathbf{u}$ to project the polytope into $\mathbf{u}^\perp$. The volume of the image of each facet $P^{\mathbf{u}_i}$ is proportional to $|\langle \mathbf{u},\mathbf{u}_i \rangle|\textrm{Vol}_{d-1}(P^{\mathbf{u}_i})$. The projection of the lower facets of $P$ with respect to $\mathbf{u}$, by which we mean the facets whose normals have negative inner product with $\mathbf{u}$, cover the projection $\pi(P)$. And the same is true for the upper facets. We can compute the volume of $\pi(P)$ using upper or lower facets, which yields
\[
\displaystyle \sum_{i: \langle \mathbf{u},-\mathbf{u}_i\rangle<0} \langle \mathbf{u},\mathbf{u}_i \rangle \textrm{Vol}_{d-1}(P^{\mathbf{u}_i}) =\sum_{i: \langle \mathbf{u},\mathbf{u}_i\rangle>0} \langle \mathbf{u},\mathbf{u}_i \rangle \textrm{Vol}_{d-1}(P^{\mathbf{u}_i}),
\]
which means that $\langle \mathbf{u},  \sum_{i=1}^n \mathbf{u}_i\textrm{Vol}_{d-1}(P^{\mathbf{u}_i})\rangle = 0 $ for all $\mathbf{u}$. This implies the result
\begin{figure}[h]
\begin{center}
\begin{tikzpicture}

\draw (0,2)--(1,3.2)--(2.3,3.4)--(4.1,2.5)--(3.3,0.8)--(1.2,0.5)--cycle;
\node at (2,2) {$P$};
\node[left] at (-2,2) {$\pi(P)$};
\node[above] at (-0.75,2) {$u$};
\draw[thin] (-2,4)--(-2,0);
\draw[dashed]  (2.3,3.4) -- (-2,3.4);
\draw[dashed]  (1.2,0.5) -- (-2,0.5);
\draw[->] (-0.5,2) -> (-1,2);
\draw[ultra thick] (-2,3.4) -- (-2,0.5);

\draw[ultra thick] (2.3,3.4)--(4.1,2.5)--(3.3,0.8)--(1.2,0.5);
\draw[ultra thick, dashed] (1.2,0.5)--(0,2)--(1,3.2)--(2.3,3.4);

\end{tikzpicture}
\end{center}
\caption{Projection of a polygon onto a line segment.}
\end{figure}
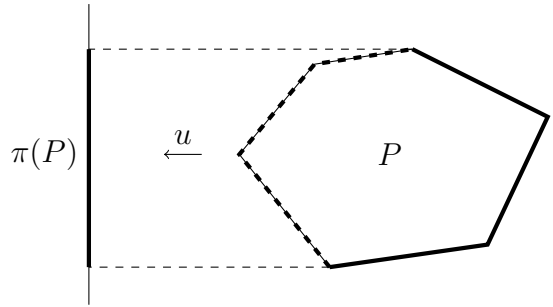

\end{proof}

In light of these relations we now give the following definition.

\begin{definition}
A $k$-balanced Minkowski weight on a polytope $P$ is a function $\omega_k:\mathcal{F}_k(P)\to\mathbb{R}$ such that for every $F\in\mathcal{F}_{k+1}(P)$ we have the following equality in the subspace $\textrm{lin}(F)$:
\[
\sum_{\substack{G\subset F \\ \textrm{a facet}}} \omega_{k}(G)\cdot \mathbf{u}_{G/F}=0.
\]
Here $\mathbf{u}_{G/F}$ is the unit outer normal in direction $G$ in the subspace $\textrm{lin}(F)$. The set of all $k$-balanced Minkowski weights on $P$ is denoted $\Omega_k(P)$, and $\Omega(P):=\bigoplus_k \Omega_k(P)$.
\end{definition}

\begin{remark}\label{rem:fin}
The set of all weak Minkowski summands of $P$ can be identified with the set of \emph{positive} 1-weights, which can be described as
\[
\left\{y\in\mathbb{R}^{f_1(P)}~:~\begin{array}{rcl}\sum_{i=1}^m y(E_i)\cdot \vec{E_i}&=0&\textrm{ if $E_1,\cdots,E_m$ form a 2-face},\\   y(E)&\geq 0&\textrm{ for all edges }E.\end{array} \right\}.
\]
This is a pointed cone in $\mathbb{R}^{f_1(P)}$, so it has finitely many rays. The rays correspond to \emph{indecomposable} polytopes $Q$, i.e., polytopes whose only weak Minkowski summands are its scalar multiples (a three dimensional example is a pyramid over a square). These correspond to the finitely many generators of the subalgebra $\Pi(P)$.
\end{remark}

Finally, McMullen uses the Minkowski map to obtain a different presentation of $\Pi(P)$.

\begin{theorem}
The Minkowski map $\phi$ induces an graded isomorphism $\Pi(P)\cong\Omega(P)$ as graded vector spaces.
\end{theorem}

McMullen's simplifies his original proof of the g-theorem in \cite{weights} by working over the ring of Minkowski weights. One technical difficulty is that we have to actually give a ring structure to $\Omega(P)$, to define, compatible with $\phi$, a way to multiply weights. This is a delicate part of \cite{weights}, but ultimately McMullen gives a simplification of the proof of the g-theorem by replacing $\Pi(P)$ with $\Omega(P)$. He writes ``The reader who wishes to work through the proof of the g-theorem in the light of the weight algebra can effectively discard much of \cite{simple}".

%Non BiBTeX users can list down their references as:

\end{document}